\documentclass[12pt]{amsart}
\usepackage{a4wide}
\usepackage{float,epsfig,hyperref,graphicx,graphics,mathrsfs,dsfont}
\usepackage{color,cite,epsfig,epstopdf,amsmath,amssymb,amsfonts}
\usepackage{pst-node}
\usepackage{tikz-cd}
\theoremstyle{definition}
\newtheorem{definition}{Definition}[subsection]

\newtheorem{theorem}{Theorem}[subsection]

\newtheorem{lemma}{Lemma}[subsection]
\newtheorem{note}{Note}[subsection]

\newtheorem{remark}{Remark}[subsection]
\title{Quantization of Cantor-like set on the real projective line}

\author{A. Hossain}
\address{Department of Mathematics, Presidency University, 86/1, College Street, Kolkata, 700 073, West Bengal, India}
\curraddr{}
\email{hossain4791@gmail.com}

\author{A. Banerjee}
\address{Department of Mathematics, Presidency University, 86/1, College Street, Kolkata, 700 073, West Bengal, India}
\email{akash.mapping@gmail.com}

\author{Md. N. Akhtar}
\address{Department of Mathematics, Presidency University, 86/1, College Street, Kolkata, 700 073, West Bengal, India}
\email{nasim.iitm@gmail.com}
\begin{document}
	\pagestyle{plain}
	\maketitle
	\begin{abstract}
		In this article, an iterated function system (IFS) is considered on the real projective line $\mathbb{RP}^1$ so that the attractor is a Cantor-like set. Hausdorff dimension of this attractor is estimated. The existence of a probability measure associated with this IFS on $\mathbb{RP}^1$ is also demonstrated. It is shown that the $n$-th quantization error of order $r$ for the push-forward measure is a constant multiple of the $n$-th quantization error of order $r$ of the original measure. Finally, an upper bound for the $n$-th quantization error of order $2$ for this measure is provided.
	\end{abstract}
	{\bf Keywords:} Real projective line, Real projective iterated function system, Probability measure, Quantization error.\\
	
	\textbf{MSC Classification}  28A80, 37A50 
\section{Introduction}
Given a probability measure on a measurable space, the quantization process involves finding a discrete set of points in the space, each point associated with a probability, such that the resulting discrete probability measure is the close approximation of the original one. Quantization theory has broad applications in signal processing, telecommunications, data compression, image processing, and cluster analysis (see \cite{Gersho2012,Gray1980,Gray1998,Zamir2014}).
Quantization error is the difference between the continuous probability measure and its discretized representation. Graf-Luschgy \cite{graf2007foundations}, studied the $n$-th quantization error for an invariant probability measure. Also, Roychowdhury \cite{roychowdhury2011quantization}, estimated the quantization dimension for the self-similar measure using the quantization error of this measure. Most of the authors studied the quantization theory in  Euclidean spaces \cite{graf1997,roychowdhury2012quantization}. Recently, Barnsley et al. \cite{Barnsley2012}, introduced the concept of real projective iterated function system (RPIFS) on the real projective space and demonstrated the existence of the attractor of this RPIFS.  Barany et al. \cite{Barany2012} studied the  Furstenberg measure generated by the RPIFS, which plays an important role in information theory. The authors used the Lyapunov exponents to determine the upper bound for the Hausdorff dimension of the Furstenberg measure on the real projective space. Also, Jurga et al. \cite{Jurga2020}, studied the dimension of the attractor of an iterated function system induced by the projective action on the real projective line. In particular, they generalized a recent result of Solomyak and Takahashi \cite{Solomyak2021} by showing that the Hausdorff dimension of the attractor is given by the minimum of 1 and the critical exponent.\\

The objective of this article is to explore an Iterated Function System (IFS) operating on the real projective line and investigate the quantization theory concerning the probability measure associated with this IFS. To achieve this, we consider an RPIFS on the real projective line $\mathbb{RP}^1$ so that it has an attractor. Then, employing methodologies outlined in prior works such as \cite{Jurga2020,Solomyak2021}, we estimate the Hausdorff dimension of this attractor. Also, we prove the existence of a probability measure associated with this RPIFS on $\mathbb{RP}^1$. Furthermore, we demonstrate that the quantization error of order $r$ for the push-forward measure is a constant multiple of the quantization error of order $r$ of the original one. We end the article by providing an upper bound of the $n$-th quantization error of order $2$ for this probability measure. 
\section{Preliminaries}
In this section, we reintroduce basic definitions of the real projective line, generating cone, projective metric, and some notations related to this article. For more details, interested readers may consult \cite{Barnsley2012,Solomyak2021,Hossain2023}.
%	\begin{definition}[Projective line]
	%		Given a Euclidean space $\mathbb{R}^2$, the projective line associated with $\mathbb{R}^2$ is the set $\mathbb{RP}^1$ of one-dimensional subspaces or (vector) lines in $\mathbb{R}^2$ (see Figure \ref{fig1}).
	%	\end{definition}
\subsection{Real projective line}
The real projective line, which is denoted by $\mathbb{RP}^1$, is the quotient of the set $\mathbb{R}^2\setminus \{(0,0)\}$ of non-zero vectors by the equivalence relation ``$x\sim y$ if and only if $x=cy$ for some $c\in\mathbb{R}^*$ (non-zero reals)". For $(x,y)\in\mathbb{R}^2\setminus \{(0,0)\}$, the corresponding class element is denoted by $[x:y]$ in $\mathbb{RP}^1$. It may be identified with the line $\mathbb{R}$ extended by a point at infinity. More precisely, the line $\mathbb{R}$ may be identified with the subset of $\mathbb{RP}^1$ given by
\begin{equation*}
	\mathbb{RP}^*=\left\{[x:1]:~x\in\mathbb{R}\right\}.
\end{equation*}
This subset covers all points of $\mathbb{RP}^1$ except one, which is the point at infinity, $\infty:=[1:0]$. Thus 
\begin{equation*}
	\mathbb{RP}^1=\mathbb{RP}^*\cup\{\infty\}.
\end{equation*}
For $[x_1:1],[x_2:1]\in\mathbb{RP}^*$, define
\begin{align*}
	&[x_1:1]\oplus[x_2:1]=[x_1+x_2:1]\\
	&[x_1:1]\star[x_2:1]=[x_1x_2:1]
\end{align*}
and the scalar multiplication of an element $[x:1]\in\mathbb{RP}^*$ with $c\in\mathbb{R}$ is defined by $c\odot[x:1]=[cx:1]$. The difference between two elements $[x_1:1],[x_2:1]\in\mathbb{RP}^*$ is defined by 
\begin{align*}
	[x_1:1]\ominus[x_2:1]=[x_1-x_2:1].
\end{align*} 
\begin{definition}[Projective metric on $\mathbb{RP}^*$]
	For $[x_1:1],[x_2:1]\in\mathbb{RP}^*$, define a metric $d_{\mathbb{P}}$ on $\mathbb{RP}^*$ as follows:
	\begin{equation}
		d_{\mathbb{P}}\big([x_1:1],[x_2:1]\big):=\vert x_1-x_2\vert.
	\end{equation} 
\end{definition}
%%%%%%%%%%%%%%%%%%%%%%%%%%%
\begin{definition}[Hyperplane]
	If  $[p],[q]\in\mathbb{RP}^1$ have the homogeneous coordinates $(p_1,p_2)$ and $(q_1,q_2)$ respectively, and $\sum_{k=1}^{{2}}p_kq_k=0$, then we say that $[p]$ is orthogonal to $[q]$, and write $[p]\perp [q]$. A  hyperplane in $\mathbb{RP}^1$ is a set of the form 
	\begin{equation*}
		\mathbb{H}_p=\big\{[q]\in\mathbb{RP}^1:[p]\perp [q]\big\}\subseteq \mathbb{RP}^1
	\end{equation*}
	for some $[p]\in\mathbb{RP}^1$.
\end{definition}
\begin{definition}[see \cite{Barnsley2012}]
	A set $\mathbb{K}\subseteq \mathbb{RP}^1$ is said to avoid a hyperplane if there exists a hyperplane $\mathbb{H}_p\subseteq \mathbb{RP}^1$ such that $\mathbb{H}_p\cap \mathbb{K}=\emptyset$.
\end{definition}
%%%%%%%%%%%%%%%%%%%%%%%%%%%%
\begin{definition}[Generating cone]
	Let $[x],[y]\in\mathbb{RP}^*$. Then the point $[x],[y]$ generates two line segments in $\mathbb{RP}^1$. We consider the segment that does not intersect the point at infinity $[1:0]$, and we denote it by $C_{\overline{xy}}$, and call it {\bf cone generated by $[x],[y]$} (see Figure \ref{gconfig}). A cone $C$ is said to be {\bf multi-cone} if $C$ is the disjoint union of finite numbers of generating cones. 
\end{definition}
\begin{figure}[ht!]
	\centering
	\includegraphics[width=8cm, height=5cm]{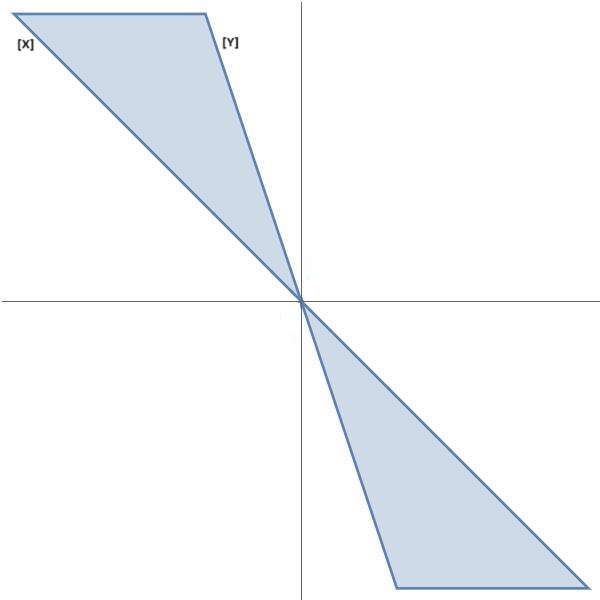}
	\caption{Cone generated by the points $[x],[y]$ in $\mathbb{RP}^*$.} \label{gconfig}
\end{figure}
\begin{definition}[Real projective iterated function system on $\mathbb{RP}^1$]
	Given a finite set $\mathcal{P}\subset GL(2,\mathbb{R})$, the associated RPIFS is denoted by $\mathscr{W}_{\mathcal{P}}=\left\{\mathbb{RP}^1; w_A:A\in\mathcal{P}\right\}$, where the projective transformations $w_A:\mathbb{RP}^1\to\mathbb{RP}^1$ are given by $w_A[x]=[Ax]$. 
	%Let $W_{\mathcal{P}}(S)=\bigcup_{A\in\mathcal{P}}w_A(S)$ for $S\in\mathbb{RP}^1$.
\end{definition}
\begin{definition}[Oriented RPIFS]
	The RPIFS $\mathscr{W}_{\mathcal{P}}=\left\{\mathbb{RP}^1; w_A:A\in\mathcal{P}\right\}$ is said to be orientation preserving if  $\mathcal{P}\subset GL^+(2,\mathbb{R})=\left\{A\in GL(2,\mathbb{R}):\det(A)>0\right\}$ or $\mathcal{P}\subset GL^-(2,\mathbb{R})=\left\{A\in GL(2,\mathbb{R}):\det(A)<0\right\}$. We denotes the corresponding RPIFS by $\mathscr{W}_{\mathcal{P}}^+$ or $\mathscr{W}_{\mathcal{P}}^-$ if $\mathcal{P}\subset GL^+(2,\mathbb{R})$ or $\mathcal{P}\subset GL^-(2,\mathbb{R})$ respectively.
\end{definition}
For simplicity, we assume that $\mathcal{P}\subset GL^+(2,\mathbb{R})$. Then the action of $GL^+(2,\mathbb{R})$ factors through the $SL(2,\mathbb{R})$ action, via $A\to \frac{A}{\sqrt{\det(A)}}$. Hence it is enough to work on $SL(2,\mathbb{R})$. 
In the sequel, we assume that $\mathcal{P}\subset SL(2,\mathbb{R})$ and we use the same notation $\mathscr{W}_{\mathcal{P}}^+$ for this case. We denote $\mathcal{P}^n$, as all the products of $n$ matrices from $\mathcal{P}$. Then $\mathcal{P}^*:=\bigcup_{n=1}^{\infty}\mathcal{P}^n$   form a semi-group generated by $\mathcal{P}$. Given a matrix $A=(a_{ij})\in\mathcal{P}$, define
\begin{equation*}
	\lVert A\rVert=\max_{ij}\{\vert a_{ij}\vert\}.
\end{equation*}
\begin{definition}
	Let $\mathcal{P}\subset SL(2,\mathbb{R})$ be finite. Then $\mathcal{P}$ is called {\bf semi-discrete} if $Id\notin \overline{\mathcal{P}^*}$, where the closure is taken over $SL(2,\mathbb{R})$. 
\end{definition}
Let $\mathcal{P}=\left\{A_i:i\in\mathcal{I}\right\}\subset SL(2,\mathbb{R})$, where $\mathcal{I}=\{1,2,\ldots,m\}$. Write $A_{\bf i}=A_{i_1}A_{i_2}\cdots A_{i_n}$ for ${\bf i}=i_1i_2\cdots i_n\in \mathcal{I}^n$.
\begin{definition}
	Let $\mathcal{P}=\left\{A_i:i\in\mathcal{I}\right\}$ be a finite collection of matrices in $SL(2,\mathbb{R})$ and $d$ be a left-invariant Riemannian metric on $SL(2,\mathbb{R})$.  Then the set $\mathcal{P}$ is said to be {\bf Diophantine} if there exists $c>0$ such that for all $n\in\mathbb{N}$, we have
	\begin{equation}
		{\bf i,j} \in \mathcal{I}^n,\quad A_{\bf i}\neq A_{\bf j}\implies	d(A_{\bf i},A_{\bf j})>c^n.
	\end{equation}
	The set $\mathcal{P}$ is said to be {\bf strongly Diophantine} if there exists $c>0$ such that for all $n\in\mathbb{N}$, we have
	\begin{equation}
		{\bf i,j} \in \mathcal{I}^n,\quad {\bf i}\neq {\bf j}\implies	d(A_{\bf i},A_{\bf j})>c^n.
	\end{equation}
\end{definition}
\begin{definition}
	A finite set $\mathcal{P}\subset SL(2,\mathbb{R})$ is called {\bf uniformly hyperbolic} if there exists $\lambda>1$ and a constant $c>0$, such that 
	\begin{equation}
		\lVert A\rVert\geq c\lambda^n,\quad\mbox{for all}~A\in\mathcal{P}^n~\mbox{and}~n\in\mathbb{N}.
	\end{equation}
\end{definition}
The next theorem can be obtained by the results from \cite{avila2010uniformly}, \cite{Barnsley2012} and \cite{Jurga2020}.
\begin{theorem}\label{extofatt}
	The following statements are equivalent.
	\begin{enumerate}
		\item The RPIFS $\mathscr{W}_{\mathcal{P}}$ has an attractor $F_{\mathcal{P}}$ that avoids a hyperplane.
		\item $\mathcal{P}$ is uniformly hyperbolic.
		\item There is a non-empty open set $V\subset\mathbb{RP}^1$ such that $W_{\mathcal{P}}$ is contractive on $\overline{V}$.
		\item There exits a multi-cone $C$ such that  $W_{\mathcal{P}}(\overline{int(C)})\subsetneq int(C)$. 
	\end{enumerate}
\end{theorem}
Given a finite or a countable set $\mathcal{P}=\left\{A_i:i\in\mathcal{I}\right\}\subset SL(2,\mathbb{R})$, define the zeta function $\zeta_{\mathcal{P}}:[0,\infty]\to\mathbb{R}\cup\{\infty\}$ by 
\begin{equation}
	\zeta_{\mathcal{P}}(t):=\sum_{n=1}^{\infty}\sum_{{\bf i}\in\mathcal{I}^n}\big(\Vert A_{\bf i}\Vert\big)^{-2t}
\end{equation}
and its critical exponent 
\begin{equation}
	\xi_{\mathcal{P}}:=\inf\left\{t>0:\zeta_{\mathcal{P}}(t)<\infty\right\}.
\end{equation}
If $\zeta_{\mathcal{P}}(t)$ is divergent for all $t\geq 0$, then define $\xi_{\mathcal{P}}=\infty$.
\begin{theorem}[\cite{Jurga2020}]\label{exthd}
	Let $\mathcal{P}=\left\{A_i:i\in\mathcal{I}\right\}$ be a finite collection of matrices in $SL(2,\mathbb{R})$ which is Diophantine and semi-discrete. If the attractor $F_{\mathcal{P}}$ of the corresponding RPIFS $\mathscr{W}_{\mathcal{P}}^+$ is not singleton, then 
	\begin{equation*}
		\dim_H(F_{\mathcal{P}})=\min\{1,\xi_{\mathcal{P}}\},
	\end{equation*}
	where $\dim_H$ denotes the Hausdorff dimension.
\end{theorem}
\begin{definition}
	A set $\mathcal{P}\subset SL(2,\mathbb{R})$ is said to be {\bf strongly irreducible}, if each map in $\mathscr{W}_{\mathcal{P}}^+$ does not preserve any finite subset of $\mathbb{RP}^1$. A set $\mathcal{P}\subset SL(2,\mathbb{R})$ is said to be {\bf irreducible} if each map in $\mathscr{W}_{\mathcal{P}}^+$ does not have a common fixed point in $\mathbb{RP}^1$.
\end{definition}
Given a finite set $\mathcal{P}=\left\{A_i:i\in\mathcal{I}\right\}\subset SL(2,\mathbb{R})$ and non-degenerate probability vector $(p_i)_{i\in\mathcal{I}}$, one can consider the probability measure $\mu$ on $SL(2,\mathbb{R})$ whose support is $\mathcal{P}$ as follows:
\begin{equation}
	\mu =\sum_{ i\in I}p_i\mathcal{X}_{A_i},
\end{equation}
where $\mathcal{X}_{A}$ denotes the characteristic function on $SL(2,\mathbb{R})$.
\begin{theorem}[\cite{Furstenberg1963,Hutchinson1981}]\label{exopbb}
	If $\mathcal{P}=\left\{A_i:i\in\mathcal{I}\right\}\subset SL(2,\mathbb{R})$ is strongly irreducible and generates an unbounded semi-group, then there exists a unique probability measure $P$ on $\mathbb{RP}^1$ such that 
	\begin{equation*}
		P=\sum_{ i\in\mathcal{I}}p_iP\circ A_i^{-1}.
	\end{equation*}
\end{theorem}
\begin{definition}[Bernoulli measure]
	Let $\mathcal{I}^*=\{1,2,\ldots,m \}^{\mathbb{N}}$ be the sequence space and for ${\bf i}=(i_1,i_2,\ldots)\in\mathcal{I}^*$, define ${\bf i}\mid_{n}=(i_1,i_2,\ldots, i_n)$. Then the cylinder set on $\mathcal{I}^*$, is defined by $[i_1,i_2,\ldots,i_n]=\left\{{\bf j}\in\mathcal{I}^*:{\bf j}\mid_{n}=(i_1,i_2,\ldots, i_n)\right\}$. Let $(p_1,p_2,\ldots,p_m)$ be the probability vector and $\mathcal{B}(\mathcal{I}^*)$ be the $\sigma$-algebra on $\mathcal{I}^*$ generated by the cylinder sets and given a cylinder set $[i_1,i_2,\ldots,i_n]$, the measure $\mu$ on $\mathcal{B}(\mathcal{I}^*)$ is defined by 
	\begin{equation}
		\mu([i_1,i_2,\ldots,i_n])=\prod_{k=1}^{n}p_{i_k}.
	\end{equation}
	This measure $\mu$ is known as {\bf Bernoulli measure}. 
\end{definition}
\begin{theorem}[\cite{Barany2012,Hutchinson1981}]\label{hutc1981}
	If $P$ is the probability measure associated with the RPIFS $\mathscr{W}_{\mathcal{P}}=\left\{\mathbb{RP}^1; w_A:A\in\mathcal{P}\right\}$ and $\mu$ be the Bernoulli measure on $\mathcal{I}^*=\{1,2,\ldots,m \}^{\mathbb{N}}$, then 
	\begin{equation}
		P=\mu\circ\Pi^{-1},
	\end{equation}
	where $\Pi:\mathcal{I}^*\to F_{\mathcal{P}}$ be the coordinate map given by $\Pi({\bf i})=\lim\limits_{n\to\infty}w_{A_{i_1}}w_{A_{i_2}}\cdots w_{A_{i_n}}([x])$, for arbitrary $[x]\in\mathbb{RP}^1$.
\end{theorem}
\section{Cantor-like set on the real projective line}
In this section, we consider the set $\mathcal{P}=\left\{A_1=\begin{pmatrix} 
	\frac{1}{3} & -\frac{2}{3} \\
	0 & 1 
\end{pmatrix}, A_2=\begin{pmatrix}
	\frac{1}{3} & \frac{2}{3} \\
	0 & 1
\end{pmatrix}\right\}\subset GL(2,\mathbb{R})$ and the corresponding RPIFS on $\mathbb{RP}^1$. First, we see that this RPIFS has an attractor. Then, we estimate the Hausdorff dimension of this attractor. We also, prove the existence of an invariant probability measure whose support is the attractor of this RPIFS. Here, we get the following results.
\begin{theorem}\label{tclset}
	The RPIFS $\mathscr{W}_{\mathcal{P}}=\left\{\mathbb{RP}^1; w_{A_1},w_{A_2}\right\}$ associated with the set $\mathcal{P}$ has an attractor.
\end{theorem}
\begin{proof}
	Since $\det(A_1)=\det(A_2)=\frac{1}{3}>0$, so, we consider the corresponding RPIFS $\mathscr{W}_{\mathcal{P}}^+$, where $\mathcal{P}=\left\{A_1=\begin{pmatrix} 
		\frac{1}{\sqrt{3}} & -\frac{2}{\sqrt{3}} \\
		0 &  \sqrt{3}
	\end{pmatrix}, A_2=\begin{pmatrix}
		\frac{1}{\sqrt{3}} & \frac{2}{\sqrt{3}}\\
		0 & \sqrt{3}
	\end{pmatrix}\right\}\subset SL(2,\mathbb{R})$. It can be seen that for all $A\in\mathcal{P}^n$ and $n\in\mathbb{N}$,
	\begin{equation}
		A=\begin{pmatrix}
			\left(\frac{1}{\sqrt{3}}\right)^n & * \\
			0 & \left(\sqrt{3}\right)^n
		\end{pmatrix}.
	\end{equation}
	So, $\Vert A\Vert\geq (\sqrt{3})^n$. Therefore, $\mathcal{P}$ is uniformly hyperbolic. Therefore, by Theorem~\ref{extofatt} the RPIFS $\mathscr{W}^+_{\mathcal{P}}=\left\{\mathbb{RP}^1; w_{A_1},w_{A_2}\right\}$ has an attractor. Since for $i=1,2$, $w_{A_i}$'s are the projective transformations. So, $\mathscr{W}^+_{\mathcal{P}}$ and $\mathscr{W}_{\mathcal{P}}$ have the same attractor. Therefore, $\mathscr{W}_{\mathcal{P}}$ has an attractor in $\mathbb{RP}^1$.
\end{proof} 
Here, we see the step-by-step construction of the attractor.\\
Let us consider the cone $C_{\overline{ab}}$, generated by the points $[a]=[-1:1]$ and $[b]=[1:1]$ (see Fig~\ref{fig3}). 
\begin{figure}[ht!]
	\centering
	\includegraphics[width=7cm, height=7cm]{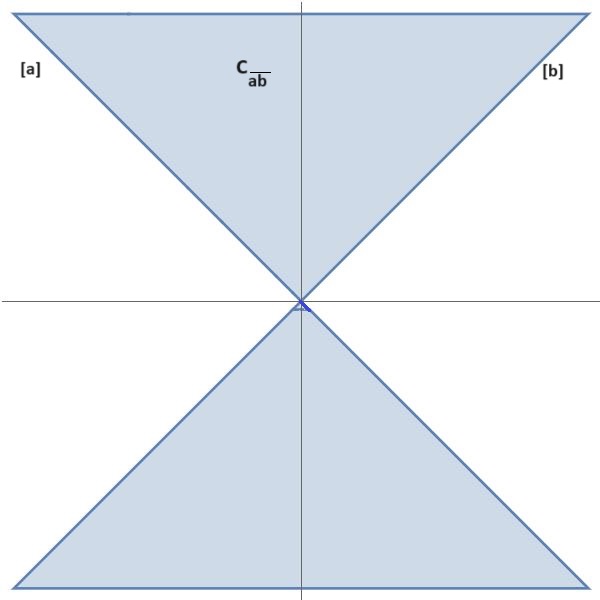}
	\caption{Cone generated by the points $[a]$ and $[b]$.} \label{fig3}
\end{figure}
Now
\begin{align*}
	&w_{A_1}[a]=[A_1a]=[-1:1];\quad w_{A_1}[b]=[A_1b]=[-\frac{1}{3}:1]\\
	&w_{A_2}[a]=[A_2a]=[\frac{1}{3}:1];\quad w_{A_2}[b]=[A_2b]=[1:1].
\end{align*}
Since projective transformation preserves the collinearity. So, $w_{A_1}(C_{\overline{ab}})=C_{\overline{a_1b_1}}$, is the cone generated by the points $[a_1]=[a]=[-1:1]$ and $[b_1]=[-\frac{1}{3}:1]$. Similarly, $w_{A_2}(C_{\overline{ab}})=C_{\overline{a_2b_2}}$, is the cone generated by the points $[a_2]=[\frac{1}{3}:1]$ and $[b_2]=[b]=[1:1]$. Therefore, $W_{\mathcal{P}}(C_{\overline{ab}})=C_{\overline{a_1b_1}}\cup C_{\overline{a_2b_2}}$ (see Fig~\ref{fig4}). To continue this process, we get the attractor of this RPIFS (see Fig~\ref{fig5}). We name it as projective Cantor set $F_{\mathcal{P}}$, (say).
\begin{figure}[ht]
	\centering
	\includegraphics[width=7cm, height=7cm]{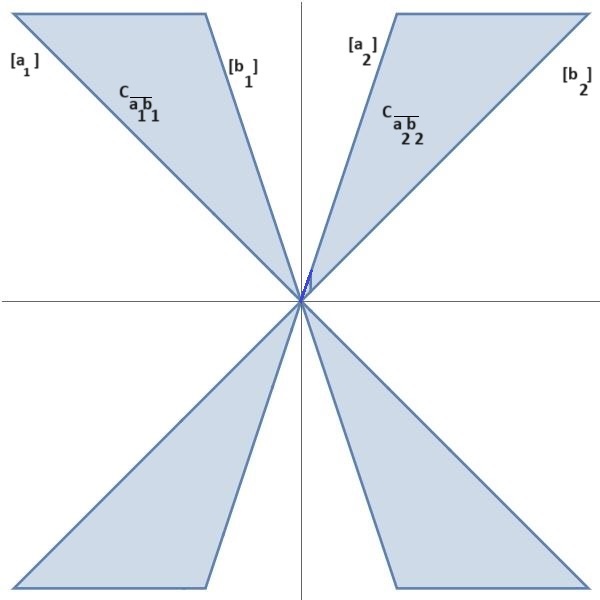}
	\caption{Multi cone generated by the points $[a_1],[b_1]$ and $[a_2],[b_2]$.} \label{fig4}
\end{figure}
\begin{figure}[ht!]
	\centering
	\includegraphics[width=7cm, height=7cm]{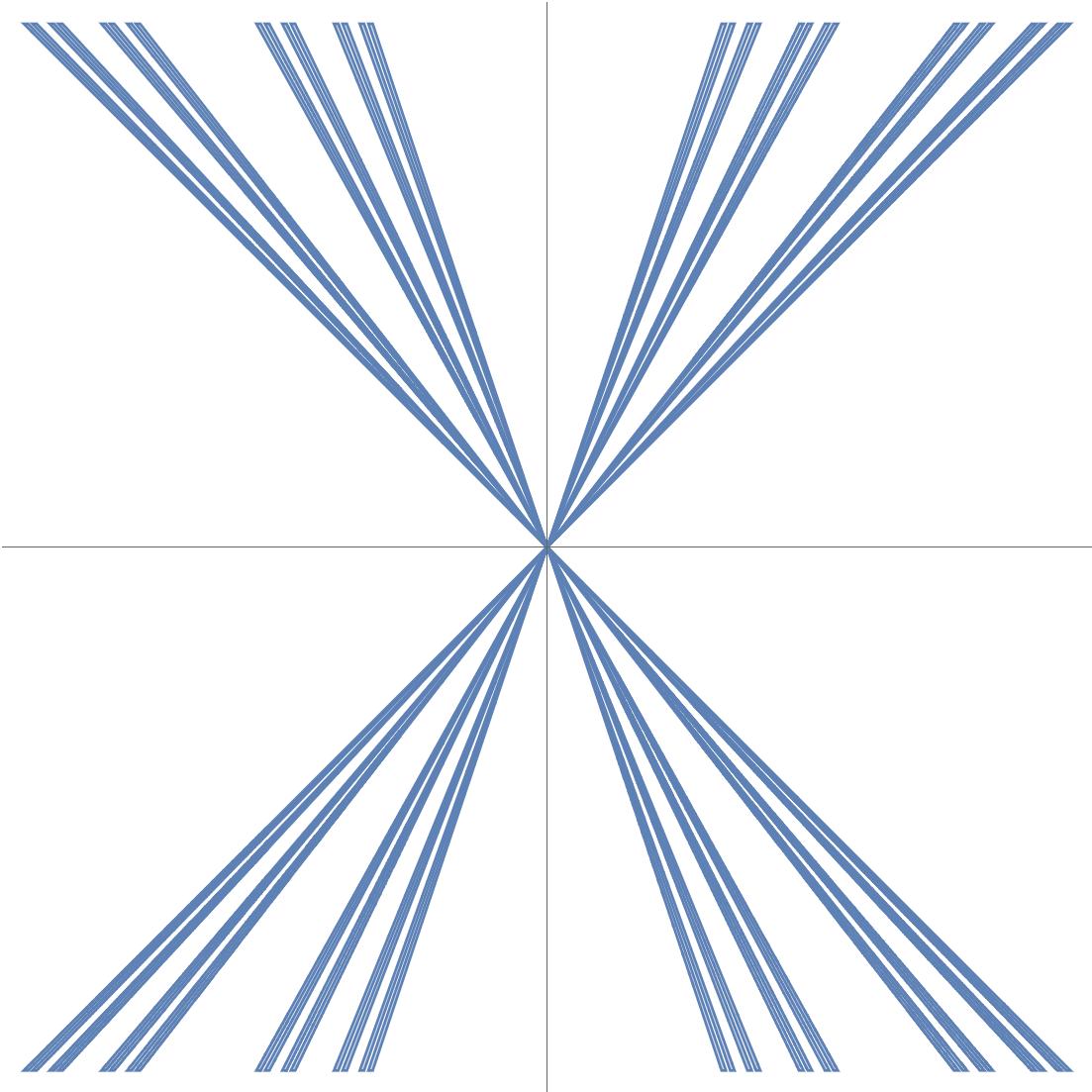}
	\caption{Projective Cantor set on $\mathbb{RP}^1$.} \label{fig5}
\end{figure}

\begin{note}
	Observe that if one see the attractor $F_{\mathcal{P}}$ on the line $y=a \;  (a\neq 0)$, then it looks like Cantor set on $\mathbb{R}$. So the attractor $F_{\mathcal{P}}$ of the RPIFS  $\mathscr{W}_{\mathcal{P}}$ is a Cantor-like set on $\mathbb{RP}^1$.
\end{note}
\begin{theorem}
	If $F_{\mathcal{P}}$ is the attractor of the RPIFS $\mathscr{W}_{\mathcal{P}}^+$ given in Theorem~\ref{tclset}, then $\dim_H(F_{\mathcal{P}})=\frac{\log 2}{\log 3}$. 
\end{theorem}
\begin{proof}
	Since 
	\begin{equation}
		{\bf i,j} \in \mathcal{I}^n,\quad {\bf i}\neq {\bf j}\implies	d(A_{\bf i},A_{\bf j})\geq \left(\frac{2}{\sqrt{3}}\right)^n.
	\end{equation}
	So the set $\mathcal{P}$ is Diophantine. Also $Id\notin \overline{A^*}$. Therefore, the set $\mathcal{P}$ is non-discrete. Also the points $[-1:1],[1:1]\in F_{\mathcal{P}}$. So, $F_{\mathcal{P}}$ is not singleton. Hence by Theorem \ref{exthd} 
	\begin{equation}\label{hausd}
		\dim_H(F_{\mathcal{P}})=\min\{1,\xi_{\mathcal{P}}\}.
	\end{equation}
	Now, all the matrices in $\mathcal{P}$ have at least one non-zero eigenvalue, so we consider the spectral norm here. That is for $A\in SL(2, \mathbb{R})$, $\Vert A\Vert=\vert\lambda\vert$, where $\lvert\lambda\rvert$ is the spectral radius of $A$. Since for any ${\bf i\in\mathcal{I}^n}$, we have 
	\begin{equation}
		A_{\bf i}=\begin{pmatrix}
			\left(\frac{1}{\sqrt{3}}\right)^n & * \\
			0 & \left(\sqrt{3}\right)^n
		\end{pmatrix}.
	\end{equation}
	So, $\Vert A_{\bf i}\Vert=\left(\sqrt{3}\right)^n$. Therefore,
	\begin{align*}
		\zeta_{\mathcal{P}}(t):=\sum_{n=1}^{\infty}\sum_{{\bf i}\in\mathcal{I}^n}\big(\Vert A_{\bf i}\Vert\big)^{-2t}&=\sum_{n=1}^{\infty}2^n\left((\sqrt{3})^n\right)^{-2t}\\
		&=\sum_{n=1}^{\infty}\left(\frac{2}{3^t}\right)^n.
	\end{align*}
	This series is convergent if $\frac{2}{3^t}< 1$. That is $\frac{\log 2}{\log 3}<t$. And it is divergent if $\frac{\log 2}{\log 3}\geq t$. Therefore, $\xi_{\mathcal{P}}=\frac{\log 2}{\log 3}$. Hence from (\ref{hausd}), $\dim_H(F_{\mathcal{P}})=\frac{\log 2}{\log 3}$. 
\end{proof}
The following theorem proves the existence of an invariant probability measure associated with $\mathscr{W}^+_{\mathcal{P}}$.
\begin{theorem}\label{methe}
	Let $p=(\frac{1}{2},\frac{1}{2})$ be the probability vector and\\ $\mathcal{P}=\left\{A_1=\begin{pmatrix} 
		\frac{1}{\sqrt{3}} & -\frac{2}{\sqrt{3}} \\
		0 &  \sqrt{3}
	\end{pmatrix}, A_2=\begin{pmatrix}
		\frac{1}{\sqrt{3}} & \frac{2}{\sqrt{3}}\\
		0 & \sqrt{3}
	\end{pmatrix}\right\}$. Let  $\mathscr{W}^+_{\mathcal{P}}=\left\{\mathbb{RP}^1; w_{A_1},w_{A_2}\right\}$ be the associated RPIFS. Then there exists unique probability measure $P$ on $\mathbb{RP}^1$ such that 
	\begin{equation}\label{pequio}
		P=\frac{1}{2}P\circ w_{A_1}^{-1}+\frac{1}{2}P\circ w_{A_2}^{-1}.
	\end{equation}
	In particular, if $\mu$ is the Bernoulli measure on the sequence space $\{1,2\}^{\mathbb{N}}$ with associated probability vector $p=(\frac{1}{2},\frac{1}{2})$, then 
	\begin{equation}\label{bmas}
		P   (E)=\mu\circ \Pi^{-1}(E)\quad\mbox{for}~E\subset\mathbb{RP}^*,
	\end{equation}
	where $\Pi:\{1,2\}^{\mathbb{N}}\to F_{\mathcal{P}}$ is given by $\Pi({\bf i})=\lim\limits_{n\to\infty}A_{i_1}A_{i_2}\cdots A_{i_n}([x])$, for arbitrary $[x]\in\mathbb{RP}^1$.
\end{theorem}
\begin{proof}
	Since $A_1$ fixes the unique point $[a]=[-1:1]$ and $A_2$ fixes the unique point $[b]=[1:1]$. So the matrices $A_1$ and $A_2$ do not preserve any finite subset of $\mathbb{RP}^1$, simultaneously. Hence the set $\mathcal{P}$ is strongly irreducible. Also, for all $A\in\mathcal{P}^n$,
	\begin{equation}
		\Vert A\Vert\geq (\sqrt{3})^n.
	\end{equation}
	So, $\mathcal{P}$ generates the unbounded semi-group. Therefore, by Theorem \ref{exopbb}, there exists a unique probability measure $P$ which satisfies the equation 
	\begin{equation*}
		P=\frac{1}{2}P\circ w_{A_1}^{-1}+\frac{1}{2}P\circ w_{A_2}^{-1}.
	\end{equation*}
	The last part of the proof follows from  Theorem~\ref{hutc1981}.
\end{proof}
%Using (\ref{bmas}), equation (\ref{pequio}) can be rewritten as 
%\begin{equation}\label{induuse}
%\mu\circ \Pi^{-1}=\frac{1}{2}\mu\circ (w_{A_1}\circ\Pi)^{-1}+\frac{1}{2}\mu\circ (w_{A_2}\circ\Pi)^{-1}.
%\end{equation}
%Using the induction on (\ref{induuse}), we get 
%\begin{equation}
%\mu\circ \Pi^{-1}=\frac{1}{2^l}\sum_{{\bf i}\in\{1,2\}^l}\mu\circ (w_{A_{\bf i}}\circ\Pi)^{-1}.
%\end{equation}
\section{Voronoi partition and Quantization error}
In this section, we define the Voronoi region on the real projective line $\mathbb{RP}^1$ and the $n$-th quantization error for a probability distribution. We see the action of a projective transformation on the Voronoi region in $\mathbb{RP}^1$. Also, we prove results related to  $n$-th quantization error of order $r$ of an invariant probability measure.
\begin{definition}
	A set $A\subset\mathbb{RP}^1$ is said to be a locally finite set in the sense that for any bounded set $B\subset\mathbb{RP}^1$, the number of elements in $A\cap B$ is finite and a collection $\mathscr{A}$ of subsets of $\mathbb{RP}^1$ is called locally finite if the number of elements in $\mathscr{A}$ intersecting any bounded subset of $\mathbb{RP}^1$ is finite.
\end{definition}
\begin{definition}
	Let $\Delta$ be a locally finite subset of $\mathbb{RP}^*$. Then the set 
	\begin{equation*}
		\mathbb{W}\left([a]\vert \Delta\right)=\left\{[x]\in\mathbb{RP}^*:d_{\mathbb{P}}\big([x],[a]\big)=\min_{[b]\in\Delta}d_{\mathbb{P}}\big([x],[b]\big)~\right\}
	\end{equation*}
	is said to be the Voronoi region generated by $[a]$. 
	%The collections $\left\{	\mathbb{W}\left([a]\vert \Delta\right): [a]\in\Delta\right\}$ is called the Voronoi diagram of $\Delta$.
\end{definition}
%A collection $\mathscr{B}$ of subsets of $\mathbb{RP}^*$ is called locally finite if the number of elements in $\mathscr{B}$ intersecting any bounded subset of $\mathbb{RP}^*$ is finite.
\begin{remark}
	The Voronoi diagram $\left\{\mathbb{W}\left([a]\vert \Delta\right): [a]\in\Delta\right\}$ is locally finite covering of $\mathbb{RP}^*$.
\end{remark}
The following result shows that the action of certain projective transformation on the Voronoi region generated by some point is same as the Voronoi region generated by the image point under the transformation. 
\begin{theorem}\label{actptonvr}
	Let $A=\begin{pmatrix}
		a_{11} & a_{12}\\
		v & 1 \\
	\end{pmatrix}\in GL(2,\mathbb{R})$ and  $T_A:\mathbb{RP}^1\to\mathbb{RP}^1$ be the corresponding projective transformation. Then the following holds:
	\begin{enumerate}
		\item If $v=0$, then $\mathbb{W}\big(T_A[a]\vert T_A(\Delta)\big)=T_A\big(\mathbb{W}\left([a]\vert \Delta\right)\big)$ for all $[a]\in\Delta$.
		\item  If $T_A(\Delta)\subset\mathbb{RP}^*$ and  $\mathbb{W}\big(T_A[a]\vert T_A(\Delta)\big)=T_A\big(\mathbb{W}\left([a]\vert \Delta\right)\big)$ for all $[a]\in\Delta$, then $v=0$.
	\end{enumerate}
\end{theorem}
\begin{proof}
	\begin{enumerate}
		\item If $v=0$. Since $\det(A)\neq 0$, so $a_{11}\neq 0$ and for all $[x:1]\in\mathbb{RP}^*$, $T_A[x:1]=[A(x,1)]=[a_{11}x+a_{12}:1]\in\mathbb{RP}^*$. Therefore, $T_A$ restricted on $\mathbb{RP}^*$  is invertible. Then 
		\begin{align*}
			\mathbb{W}\big(T_A[a]\vert T_A(\Delta)\big)&=\left\{[x]\in\mathbb{RP}^*:d_{\mathbb{P}}\big([x],T_A[a]\big)=\min_{[b]\in T_A(\Delta)}d_{\mathbb{P}}\big([x],[b]\big)\right\}\\
			&=\left\{T_A[x]\in\mathbb{RP}^*:d_{\mathbb{P}}\big(T_A[x],T_A[a]\big)=\min_{[d]\in \Delta}d_{\mathbb{P}}\big(T_A[x],T_A[d]\big)\right\}.
		\end{align*}
		Now, if $[c]=[c:1]\in\mathbb{RP}^*$, then
		\begin{align*}
			d_{\mathbb{P}}\big(T_A[x],T_A[c]\big)&=d_{\mathbb{P}}\big([a_{11}x+a_{12}:1],[a_{11}c+a_{12}:1]\big)\\
			&=\vert a_{11}\vert \vert x-c\vert\\
			&=\vert a_{11}\vert d_{\mathbb{P}}\big([x],[c]\big).
		\end{align*} 
		Therefore,
		\begin{align*}
			\mathbb{W}\big(T_A[a]\vert T_A(\Delta)\big)&=\left\{T_A[x]\in\mathbb{RP}^*:\vert a_{11}\vert d_{\mathbb{P}}\big([x],[a]\big)=\min_{[d]\in \Delta}\vert a_{11}\vert d_{\mathbb{P}}\big([x],[d]\big)\right\}\\
			&=\left\{T_A[x]\in\mathbb{RP}^*: d_{\mathbb{P}}\big([x],[a]\big)=\min_{[d]\in \Delta} d_{\mathbb{P}}\big([x],[d]\big)\right\}\\
			&=T_A\big(\mathbb{W}\left([a]\vert \Delta\right)\big).
		\end{align*}
		\item  If possible let $v\neq 0$. Then $[-\frac{1}{v}:1]\in\mathbb{RP}^*$. Since $\mathbb{RP}^*= \bigcup_{[a]\in\Delta}\mathbb{W}\left([a]\vert \Delta\right)$, so, $[-\frac{1}{v}:1]\in \mathbb{W}\left([c]\vert \Delta\right)$  for some $[c]\in \Delta$. Therefore, $T_A[-\frac{1}{v}:1]=[-\frac{1}{v}a_{11}+a_{12}:0]\in T_A\big(\mathbb{W}\left([c]\vert \Delta\right)\big)$. Given that $T_A(\Delta)\subset\mathbb{RP}^*$, so $[a]\neq [-\frac{1}{v}:1]$ for all $[a]\in\Delta$. Therefore, $\mathbb{RP}^*= \bigcup_{[a]\in\Delta}\mathbb{W}\left(T_A[a]\vert T_A(\Delta)\right)$. Hence $[-\frac{1}{v}:1]\notin \mathbb{W}\left(T_A[c]\vert T_A(\Delta)\right)$. This gives a contradiction to $\mathbb{W}\big(T_A[c]\vert T_A(\Delta)\big)=T_A\big(\mathbb{W}\left([c]\vert \Delta\right)\big)$. Hence the proof.
	\end{enumerate}
\end{proof} 
\begin{definition}
	Let $\mu$ be the probability distribution on $\mathbb{RP}^*$ and $[x:1]$ be the random variable on $\mathbb{RP}^*$ with distribution $\mu$. Let 
	\begin{equation*}
		\mathscr{F}_n=\left\{f:\mathbb{RP}^*\to\mathbb{RP}^*~\mbox{such that $f$ is measurable map with $\vert f\big(\mathbb{RP}^*\big)\vert\leq n$}\right\}.
	\end{equation*}
	The elements of $\mathscr{F}_n$ are called $n$-quantizers. Let $r\geq 1$ and assume that 
	\begin{equation*}
		\int_{\mathbb{RP}^*}d_{\mathbb{P}}\big([x:1],[0:1]\big)^rd\mu[x]<\infty.
	\end{equation*}
	The $n$-th quantization error for $\mu$ of order $r$ is defined by
	\begin{equation*}
		\mathbb{V}_{n,r}(\mu)=\inf_{f\in\mathscr{F}_n}\int_{\mathbb{RP}^*}d_{\mathbb{P}}\big([x],f[x]\big)^rd\mu[x].
	\end{equation*}
	A quantizer $f\in\mathscr{F}_n$ is said to be $n$-optimal for $\mu$ of order $r$ if 
	\begin{equation*}
		\mathbb{V}_{n,r}(\mu)=\int_{\mathbb{RP}^*}d_{\mathbb{P}}\big([x],f[x]\big)^rd\mu[x].
	\end{equation*}
\end{definition}
The following lemma shows that the $n$-quantizer can be replaced by any finite set of cardinality less than or equal to $n$.
\begin{lemma}
	\begin{equation*}
		\mathbb{V}_{n,r}(\mu)=\displaystyle{\inf_{\mbox{\begin{tabular}{c} $\Delta\subset\mathbb{RP}^*$\\ 
						$\vert\Delta\vert<n$ \end{tabular}}}
			\int_{\mathbb{RP}^*}\min_{[a]\in\Delta}d_{\mathbb{P}}\big([x],[a]\big)^rd\mu[x]}.
	\end{equation*}
\end{lemma}
\begin{proof}
	Let $\Delta=f\big(\mathbb{RP}^*\big)$ for a fixed $f\in\mathscr{F}_n$. Then $\vert \Delta\vert\leq n$. For $[a]\in\Delta$, let 
	\begin{equation*}
		\mathbb{A}_{[a]}=\left\{[x]\in\mathbb{RP}^*:~f[x]=[a]\right\}.
	\end{equation*}
	Then $\mathbb{A}_{[a]}$ is non-empty and $\bigcup_{[a]\in\Delta}\mathbb{A}_{[a]}=\mathbb{RP}^*$.
	Now,
	\begin{align*}
		\int_{\mathbb{RP}^*}d_{\mathbb{P}}\big([x],f[x]\big)^rd\mu[x]&=\sum_{[a]\in\Delta}\int_{\mathbb{A}_{[a]}}d_{\mathbb{P}}\big([x],[a]\big)^rd\mu[x]\\
		&\geq \sum_{[a]\in\Delta}\int_{\mathbb{A}_{[a]}}\min_{[b]\in\Delta}d_{\mathbb{P}}\big([x],[b]\big)^rd\mu[x]\\
		&=\int_{\mathbb{RP}^*}\min_{[b]\in\Delta}d_{\mathbb{P}}\big([x],[b]\big)^rd\mu[x].
	\end{align*}
	Therefore, 
	\begin{equation}\label{exp1}
		\mathbb{V}_{n,r}(\mu)=\inf_{f\in\mathscr{F}_n}\int_{\mathbb{RP}^*}d_{\mathbb{P}}\big([x],f[x]\big)^rd\mu[x]\geq\displaystyle{\inf_{\mbox{\begin{tabular}{c} $\Delta\subset\mathbb{RP}^*$\\ 
						$\vert\Delta\vert<n$ \end{tabular}}}
			\int_{\mathbb{RP}^*}\min_{[a]\in\Delta}d_{\mathbb{P}}\big([x],[a]\big)^rd\mu[x]}.
	\end{equation}
	Conversely, let $\Delta\subset\mathbb{RP}^*$ with $\vert\Delta\vert\leq n$ and $\mathbb{A}_{[a]}=\mathbb{W}\left([a]\vert \Delta\right)$. For $[a]\in\Delta$, define $f:\mathbb{RP}^*\to\mathbb{RP}^*$ such that $f=\sum_{[a]\in\Delta}[a]\chi_{\mathbb{A}_{[a]}}$, where the summation is over $\oplus$ and $\chi_{\mathbb{A}_{[a]}}:\mathbb{RP}^*\to\mathbb{RP}^*$ is the characteristic function on $\mathbb{A}_{[a]}$. Then $f\in\mathscr{F}_n$ and
	\begin{align*}
		\int_{\mathbb{RP}^*}\min_{[a]\in\Delta}d_{\mathbb{P}}\big([x],[a]\big)^rd\mu[x]&=\sum_{[a]\in\Delta}\int_{\mathbb{A}_{[a]}}\min_{[a]\in\Delta}d_{\mathbb{P}}\big([x],[a]\big)^rd\mu[x].
	\end{align*}
	Sincc $[x]\in \mathbb{A}_{[a]}=\mathbb{W}\left([a]\vert \Delta\right)$. Therefore,
	\begin{align*}
		\int_{\mathbb{RP}^*}\min_{[a]\in\Delta}d_{\mathbb{P}}\big([x],[a]\big)^rd\mu[x]&=\sum_{[a]\in\Delta}\int_{\mathbb{A}_{[a]}}d_{\mathbb{P}}\big([x],[a]\big)^rd\mu[x]\\
		&=\int_{\mathbb{RP}^*}d_{\mathbb{P}}\big([x],f[x]\big)^rd\mu[x]\\
		&\geq\inf_{f\in\mathscr{F}_n}\int_{\mathbb{RP}^*}d_{\mathbb{P}}\big([x],f[x]\big)^rd\mu[x]=\mathbb{V}_{n,r}(\mu).
	\end{align*}
	So,
	\begin{equation}\label{expofdmj}
		\displaystyle{\inf_{\mbox{\begin{tabular}{c} $\Delta\subset\mathbb{RP}^*$\\ 
						$\vert\Delta\vert<n$ \end{tabular}}}
			\int_{\mathbb{RP}^*}\min_{[a]\in\Delta}d_{\mathbb{P}}\big([x],[a]\big)^rd\mu[x]}\geq \mathbb{V}_{n,r}(\mu).
	\end{equation}
	Hence from (\ref{exp1}) and (\ref{expofdmj}), we get the desired result.
\end{proof}
\begin{definition}
	A set $\Delta\subset\mathbb{RP}^*$ with $\vert\Delta\vert<n$ is said to be $n$-optimal set of centers for the probability distribution $\mu$ of order $r$ if 
	\begin{align*}
		\mathbb{V}_{n,r}(\mu)=
		\int_{\mathbb{RP}^*}\min_{[a]\in\Delta}d_{\mathbb{P}}\big([x],[a]\big)^rd\mu[x].
	\end{align*}
\end{definition}
Let $\mathscr{C}_{n,r}(\mu)$ be the collection of all $n$-optimal set of centers for $\mu$ of order $r$. Then the following result is concerning with the $n$-th quantization error and optimal set of the measure $\mu$ and push-forward measure $T_A\mu$ under certain projective transformation $T_A$.
\begin{theorem}
	Let $A=\begin{pmatrix}
		a_{11} & a_{12}\\
		v & 1 \\
	\end{pmatrix}\in GL(2,\mathbb{R})$ and  $T_A:\mathbb{RP}^1\to\mathbb{RP}^1$ be the corresponding projective transformation. Then the following holds:
	\begin{enumerate}
		\item If $v=0$, then	$\mathbb{V}_{n,r}(T_A\mu)=\vert a_{11}\vert^r~\mathbb{V}_{n,r}(\mu)$.
		\item  If $v=0$, then $\mathscr{C}_{n,r}(T_A\mu)=T_A\left(\mathscr{C}_{n,r}(\mu)\right)$, 
	\end{enumerate}	
	where $T_A\mu$ is the push-forward measure of $\mu$.
\end{theorem}
\begin{proof}
	\begin{enumerate}
		\item If $v=0$. Let $\Omega\subset T_A\big(\mathbb{RP}^*\big)$ such that $\vert\Omega\vert\leq n$. Since $v=0$, so, $T_A\vert_{\mathbb{RP}^*}$ is a non-singular map on $\mathbb{RP}^*$. Therefore, there is $\Delta$ in $\mathbb{RP}^*$ such that $T_A(\Delta)=\Omega$ and $\vert\Delta\vert=\vert\Omega\vert\leq n$. Then
		\begin{align*}
			\int_{T_A(\mathbb{RP}^*)}\min_{[a]\in\Omega}d_{\mathbb{P}}\big([y],[a]\big)^rd\left(\mu\circ T_A^{-1}\right)[y]&=\int_{\mathbb{RP}^*}\min_{[b]\in\Delta}d_{\mathbb{P}}\big(T_A[x],T_A[b]\big)^rd\mu[x]\\
			&=\vert a_{11}\vert^r\int_{\mathbb{RP}^*}\min_{[b]\in\Delta}d_{\mathbb{P}}\big([x],[b]\big)^rd\mu[x]\\
			&\geq \vert a_{11}\vert^r~\mathbb{V}_{n,r}(\mu).
		\end{align*}
		Therefore, taking infimum over $\Omega$ with $\vert\Omega\vert\leq n$, we get
		\begin{align}\label{qeudgsmg}
			\mathbb{V}_{n,r}(T_A\mu)\geq\vert a_{11}\vert^r~\mathbb{V}_{n,r}(\mu).
		\end{align}
		Now, it is an easy exercise to see that $T_A^{-1}=T_{A^{-1}}$, where $A^{-1}=\begin{pmatrix}
			\frac{1}{a_{11}} & -\frac{a_{12}}{a_{11}}\\
			0 & 1 \\
		\end{pmatrix}.$
		So, if we replace $T_A$ by $T_A^{-1}$ in the beginning and proceed as above then we get
		\begin{align*}
			\mathbb{V}_{n,r}(T^{-1}_A\mu)\geq \frac{1}{\vert a_{11}\vert^r}~\mathbb{V}_{n,r}(\mu).
		\end{align*}
		That is 
		\begin{align}\label{resqeudgsmg}
			\vert a_{11}\vert^r\mathbb{V}_{n,r}(\mu)\geq~\mathbb{V}_{n,r}(T_A\mu).
		\end{align}
		Combining (\ref{qeudgsmg}) and (\ref{resqeudgsmg}), result follows.
		\item Let $\Omega\subset T_A\big(\mathbb{RP}^*\big)$ such that $\Omega\in\mathscr{C}_{n,r}(T_A\mu)$. Since $T_A$ is invertible and $v=0$, so, there is $\Delta$ in $\mathbb{RP}^*$ such that $\Omega=T_A(\Delta)$ and $\vert\Delta\vert=\vert\Omega\vert\leq n$. We claim that $\Delta$ is also an $n$-optimal set of the centers for $\mu$. Since $\Omega\in\mathscr{C}_{n,r}(T_A\mu)$, therefore, 
		\begin{align}\label{pusfrd}
			\nonumber	\mathbb{V}_{n,r}(T_A\mu)&=\int_{T_A(\mathbb{RP}^*)}\min_{[a]\in\Omega}d_{\mathbb{P}}\big([y],[a]\big)^rd\left(\mu\circ T_A^{-1}\right)[y]\\
			\nonumber	&=\int_{\mathbb{RP}^*}\min_{[b]\in\Delta}d_{\mathbb{P}}\big(T_A[x],T_A[b]\big)^rd\mu[x]\\
			&=\vert a_{11}\vert^r\int_{\mathbb{RP}^*}\min_{[b]\in\Delta}d_{\mathbb{P}}\big([x],[b]\big)^rd\mu[x].
		\end{align}
		Since $a_{11}\neq 0$, and from the above we have $\mathbb{V}_{n,r}(T_A\mu)=\vert a_{11}\vert^r~\mathbb{V}_{n,r}(\mu)$. Therefore from (\ref{pusfrd}), it follows that
		\begin{align*}
			\mathbb{V}_{n,r}(\mu)=\int_{\mathbb{RP}^*}\min_{[b]\in\Delta}d_{\mathbb{P}}\big([x],[b]\big)^rd\mu[x].
		\end{align*}
		This proves our claim. Hence $\Delta\in\mathscr{C}_{n,r}(\mu)$. Therefore, $\Omega=T_A(\Delta)\in T_A\left(\mathscr{C}_{n,r}(\mu)\right)$. This implies that $\mathscr{C}_{n,r}(T_A\mu)\subseteq T_A\left(\mathscr{C}_{n,r}(\mu)\right)$. To prove the converse part, we replace $T_A$ by $T_A^{-1}$ in the beginning and proceed as above. This completes the proof.
	\end{enumerate}
\end{proof}

To estimate the upper bound of the $n$-th quantization error of a probability measure, we recall the RPIFS $\mathscr{W}_{\mathcal{P}}=\left\{\mathbb{RP}^1; w_{A_1},w_{A_2}\right\}$ and the associated invariant probability measure $P$ defined as above. Let $C_{\overline{ab}}$, be the cone generated by the points $[a]=[-1:1]$ and $[b]=[1:1]$. For $\omega=(\omega_1,\omega_2,\ldots,\omega_l)\in\left\{1,2\right\}^l$, define $w_{A_{\omega}}=w_{A_{\omega_1}}\circ w_{A_{\omega_2}}\circ\cdots\circ w_{A_{\omega_l}}$. Let $C_{\omega}=w_{A_{\omega}}(C_{\overline{ab}})$ and $[c]$, $[d]$ be the boundary points of the cone $C_{\omega}$. Let $m_{\omega}\in C_{\omega}$ be the point which lies in the equal distance from the points $[c]$ and $[d]$ concerning the metric $d_{\mathbb{P}}$. That is the mid-point of the cone $C_{\omega}$. The collection $\left(C_{\omega}\right)_{\omega\in\left\{1,2\right\}^l}$ consists $2^l$ number of cones in the $l$-th level of the construction of Cantor-like set on $\mathbb{RP}^1$. The cones $C_{\omega1^*}$, $C_{\omega2^*}$ are said to be children of $C_{\omega}$ into which $C_{\omega}$ is split up at $(l+1)$-th level. Then the attractor $F_{\mathbb{P}}=\displaystyle{\bigcap_{l\in\mathbb{N}}\bigcup_{\omega\in\left\{1,2\right\}^l}C_{\omega}}$ is the Cantor-like set on $\mathbb{RP}^1$.\\

Let $n\in\mathbb{N}$ with $n\geq 1$, and $k(n)\in\mathbb{N}$ be the unique number such that $2^{k(n)}\leq n< 2^{k(n)+1}$. Let $\mathbb{I}\subset \left\{1,2\right\}^{k(n)}$ with $card(\mathbb{I})=n-2^{k(n)}$, and let $\Delta_n$ be the set consisting all points $m_{\omega}$ of the cone $C_{\omega}$ with $\omega\in\left\{1,2\right\}^{k(n)}\setminus \mathbb{I}$ and all mid-points $m_{\omega1^*}$, $m_{\omega2^*}$ of the children of $C_{\omega}$ with $\omega\in\mathbb{I}$. That is 
\begin{equation*}
	\Delta_n=\left\{m_{\omega}:\omega\in\left\{1,2\right\}^{k(n)}\setminus \mathbb{I}\right\}\cup\left\{m_{\omega1^*}:\omega\in\mathbb{I}\right\}\cup\left\{m_{\omega2^*}:\omega\in\mathbb{I}\right\}.
\end{equation*}
Let $D_n=\frac{1}{2}\frac{1}{18^{k(n)}}\left[2^{k(n)+1}-n+\frac{1}{9}(n-2^{k(n)})\right]$. Then we get the following:
\begin{theorem}\label{qnlman}
	The upper bound of the $n$-th quantization error for $P$ of order $2$ is  $D_n.$
\end{theorem}
To prove the above theorem, first, we prove the following lemmas.
\begin{lemma}\label{upmrs}
	If $S:\mathbb{RP}^*\to\mathbb{R}_{+}$ is  Borel measurable and $k\in\mathbb{N}$, then 
	\begin{equation*}
		\int S[x] dP[x]=\frac{1}{2^l}\sum_{{\bf i}\in\{1,2\}^l}\int S\circ w_{A_{\bf i}}[x]dP[x].
	\end{equation*}
\end{lemma}
\begin{proof}
	From (\ref{pequio}),
	\begin{align}
		\nonumber    \int S[x] dP[x]&=\int S[x] d\big(\frac{1}{2}P\circ w_{A_1}^{-1}+\frac{1}{2}P\circ w_{A_2}^{-1}\big)[x]\\
		\nonumber    &=\frac{1}{2}\sum_{{\bf i}\in\{1,2\}}\int S[x]dP\circ w^{-1}_{A_{\bf i}}[x]\\
		\label{lastoneryjhd}   &=\frac{1}{2}\sum_{{\bf i}\in\{1,2\}}\int S\circ w_{A_{\bf i}}[x]dP[x].
	\end{align}
	Now, by repeated applications of the equation (\ref{pequio}) in (\ref{lastoneryjhd}), we get the required result.
\end{proof}	
\begin{lemma}\label{distcpot}
	For  $[0:1]\in\mathbb{RP}^*$,
	\begin{align*}
		\int d_{\mathbb{P}}\big([x:1],[0:1]\big)^2dP[x]=\frac{1}{2}.
	\end{align*}
\end{lemma}
\begin{proof}
	For $[0:1]\in\mathbb{RP}^*$, using lemma~\ref{upmrs}
	\begin{align*}
		\int d_{\mathbb{P}}\big([x:1],[0:1]\big)^2dP[x]&=\frac{1}{2}\int d_{\mathbb{P}}\big(w_{A_{1}}[x:1],[0:1]\big)^2dP[x]+\frac{1}{2}\int d_{\mathbb{P}}\big(w_{A_{2}}[x:1],[0:1]\big)^2dP[x]\\
		&=\frac{1}{2}\int d_{\mathbb{P}}\big([\frac{x}{3}-\frac{2}{3}:1],[0:1]\big)^2dP[x]+\frac{1}{2}\int d_{\mathbb{P}}\big([\frac{x}{3}+\frac{2}{3}:1],[0:1]\big)^2dP[x]\\
		&=\frac{1}{2}\left[\int \left(\frac{x}{3}-\frac{2}{3}\right)^2dP[x]+\int \left(\frac{x}{3}+\frac{2}{3}\right)^2dP[x]\right]\\
		&=\frac{1}{2}\left[2\int \frac{x^2}{9}dP[x]+2\int \frac{4}{9}dP[x]\right]\\
		&=\frac{1}{9}\int x^2dP[x]+\frac{4}{9}\\
		&=\frac{1}{9}\int d_{\mathbb{P}}\big([x:1],[0:1]\big)^2dP[x]+\frac{4}{9}.
	\end{align*}
	Therefore, 
	\begin{equation*}
		\int d_{\mathbb{P}}\big([x:1],[0:1]\big)^2dP[x]=\frac{1}{2}.
	\end{equation*}
\end{proof}
\begin{proof}[Proof of  Theorem~\ref{qnlman}]
	\begin{align*}
		\int_{\mathbb{RP}^*}\min_{[a]\in\Delta_n}d_{\mathbb{P}}\big([x],[a]\big)^2dP[x]=&\sum_{\omega\in\{1,2\}^{k(n)}\setminus\mathbb{I}}\int_{C_{\omega}}d_{\mathbb{P}}\big([x],[a]\big)^2dP[x]+\sum_{\omega\in\mathbb{I}}\int_{C_{\omega1^*}}d_{\mathbb{P}}\big([x],[a]\big)^2dP[x]\\
		&+\sum_{\omega\in\mathbb{I}}\int_{C_{\omega2^*}}d_{\mathbb{P}}\big([x],[a]\big)^2dP[x].
	\end{align*}
	For every $l\in\mathbb{N}$ and for every $\sigma\in\left\{1,2\right\}^{l}$, from Lemma~\ref{upmrs} and Lemma~\ref{distcpot}, we get
	\begin{align*}
		\int_{C_\sigma}d_{\mathbb{P}}\big([x],[m_{\sigma}]\big)^2dP[x]&=\frac{1}{2^l}\int d_{\mathbb{P}}\big(w_{A_{\sigma}}[x],[m_{\sigma}]\big)^2dP[x]\\
		&=\frac{1}{2^l}\int d_{\mathbb{P}}\big(w_{A_{\sigma}}[x:1],w_{A_{\sigma}}[0:1]\big)^2dP[x].
	\end{align*} 
	Now, $w_{A_{\sigma}}=\begin{pmatrix}
		\frac{1}{3^l} & * \\
		0 & 1
	\end{pmatrix}.$ Therefore, $d_{\mathbb{P}}\big(w_{A_{\sigma}}[x:1],w_{A_{\sigma}}[0:1]\big)^2=\left(\frac{1}{3^l}\right)^2d_{\mathbb{P}}\big([x:1],[0:1]\big)^2$. So,
	\begin{align*}
		\int_{C_\sigma}d_{\mathbb{P}}\big([x],[m_{\sigma}]\big)^2dP[x]&=\frac{1}{18^l}\int d_{\mathbb{P}}\big([x:1],[0:1]\big)^2dP[x]\\                                   &=\frac{1}{2}\frac{1}{18^l}.
	\end{align*}
	Since $\omega1^*$ and $\omega2^*$ are the elements of $\left\{1,2\right\}^{k(n)+1}$, therefore,
	\begin{align*}
		\int_{\mathbb{RP}^*}\min_{[a]\in\Delta_n}d_{\mathbb{P}}\big([x],[a]\big)^2dP[x]=&\frac{1}{2}\frac{1}{18^{k(n)}}\sum_{\omega\in\{1,2\}^{k(n)}\setminus\mathbb{I}}\int_{C_{\omega}}dP[x]+\frac{1}{2}\frac{1}{18^{k(n)+1}}\sum_{\omega\in\mathbb{I}}\int_{C_{\omega1^*}}dP[x]\\
		&+\frac{1}{2}\frac{1}{18^{k(n)+1}}\sum_{\omega\in\mathbb{I}}\int_{C_{\omega2^*}}dP[x]\\
		=&\frac{1}{2}\frac{1}{18^{k(n)}}\left[card\big(\{1,2\}^{k(n)}\setminus\mathbb{I}\big)+\frac{1}{9}card\big(\mathbb{I}\big)\right]\\
		=&\frac{1}{2}\frac{1}{18^{k(n)}}\left[2^{k(n)}-(n-2^{k(n)})+\frac{1}{9}(n-2^{k(n)})\right]\\
		=&\frac{1}{2}\frac{1}{18^{k(n)}}\left[2^{k(n)+1}-n+\frac{1}{9}(n-2^{k(n)})\right]=D_n.
	\end{align*}
	Hence 
	\begin{align*}
		\mathbb{V}_{n,r}(P)\leq D_n.
	\end{align*}
	This completes the proof.
\end{proof}

\section*{Conclusion}
In this article, an Iterated Function System is considered on the real projective line $\mathbb{RP}^1$ so that the attractor is a Cantor-like set on it, and we estimated the Hausdorff dimension of such an attractor. This involves analyzing the geometric and fractal properties of the attractor, providing a detailed understanding of its dimensional characteristics. Then, we have shown the existence of a probability measure on $\mathbb{RP}^1$ associated with this IFS and estimated the upper bound of the $n$-th quantization error for this measure. This quantization error gives insight into the efficiency of representing the probability measure with a finite set of points, which has applications in data compression and signal processing. In the future, one may find the optimal set and estimate the quantization dimension of the Cantor-like set in 
$\mathbb{RP}^1$. This future work would involve finding the set of points that minimize the quantization error and determining the quantization dimension, which would provide a deeper understanding of the fine structure of the attractor and its representation in 
$\mathbb{RP}^1$. This work contributes to understanding quantization theory on the non-Euclidean projective spaces.

	\bibliographystyle{siamplain}
	\bibliography{QuantizationProjectiveLine}	
\end{document}